\documentclass{article}

\usepackage{geometry} 
\usepackage{verbatim}
\usepackage{amsmath}
\usepackage{makecell}
\usepackage{graphicx}
\usepackage{tikz-cd}
\usepackage{tikz}
\usepackage{mathtools}
\usepackage{xfrac}
\usepackage{color}
\usepackage{amsfonts, amssymb, amsthm,  latexsym}
\usepackage{multicol}
\usepackage{graphicx}
\usepackage[all,cmtip,curve, knot,frame]{xy} 
\usepackage[at]{easylist}
\usepackage{subfiles}
\usepackage{enumitem}
\usepackage{setspace}
\usepackage{todonotes}
\usepackage{scalefnt}
\usepackage{centernot}

\usepackage{enumitem}
\usepackage{hhline}
\usepackage[utf8]{inputenc}
\usepackage[english]{babel}
\usepackage{biblatex}

\usepackage{xcolor}
\usetikzlibrary{knots}
\usetikzlibrary{decorations.markings}
\usepackage{hyperref}
\usepackage{csquotes}
\graphicspath{ {./figures/} }
\numberwithin{equation}{section}
\newtheorem*{theorem*}{Theorem}
\newtheorem{theorem}{Theorem}[section]
\newtheorem{theorem/def}[theorem]{Theorem/Definition}
\newtheorem{lemma}[theorem]{Lemma}
\newtheorem{proposition}[theorem]{Proposition}

\newtheorem{conjecture}[theorem]{Conjecture}
\newtheorem{corollary}[theorem]{Corollary}
\theoremstyle{definition}
\newtheorem{definition}[theorem]{Definition}
\newtheorem{remark}[theorem]{Remark}
\newtheorem{notation}[theorem]{Notation}
\newtheorem{example}[theorem]{Example}
\newtheorem{def/prop}[theorem]{Definition/Proposition}

\newcommand{\C}{\mathbb{C}}

\newcommand{\Z}{\mathbb{Z}}

\newcommand{\K}{\mathcal{K}}
\newcommand{\cA}{\mathcal{A}}
\newcommand{\A}[2]{A_{#2 #1}}
\newcommand{\HH}
{\operatorname{HH}_0}
\newcommand{\defterm}[1]{\textbf{#1}}
\newcommand{\HHA}[2]{\HH(A,\A{#1}{#2})}
\newcommand{\SL}[1]{\operatorname{SL}_{#1}}
\newcommand{\SLZ}{\SL{2}(\Z)}
\newcommand{\GL}[1]{\operatorname{GL}_{#1}}
\newcommand{\GLN}{\GL{N}}
\newcommand{\SLN}{\SL{N}}

\newcommand{\s}[1]{\operatorname{r}_{#1}}
\newcommand{\SkMod}[1]{\operatorname{Sk}_{#1}}
\newcommand{\Sk}[1]{\SkMod{\GL{#1}}}
\newcommand{\SkAlg}[1]{\operatorname{SkAlg}_{\GL{#1}}}
\newcommand{\SkCat}[1]{\operatorname{SkCat}_{#1}}
\newcommand{\M}[3]{\Id_{2 #3}-#2 \otimes #1}

\newcommand{\Mwg}{\M{\gamma}{w}{N}}
\newcommand{\Mred}[2]{\Id-#1^{#2}}
\newcommand{\cycle}[1]{(1\cdots #1)}

\newcommand{\cycletyperep}[1]{1^{\s{1}}2^{\s{2}}\ldots #1^{\s{#1}}}
\newcommand{\cycletype}[1]{[\cycletyperep{#1}]}
\newcommand{\cyclicgr}[1]{C_{#1}}
\newcommand{\rkw}{\s{k}(w)}
\newcommand{\rjw}{\s{j}(w)}
\newcommand{\ckg}{c_k(\gamma)}
\newcommand{\ck}{c_k}
\newcommand{\quotient}[2]{{\raisebox{.0em}{$#1$}\left/\raisebox{-.15em}{$#2$}\right.}}
\newcommand{\colred}{\begin{array}{cc}
    \cong\\
    \text{col}
\end{array}}
\newcommand{\rowred}{\begin{array}{cc}
    \cong\\
    \text{row}
\end{array}}
\newcommand{\cycgp}[2]{\quotient{\langle#1\rangle}{\langle#1^{#2}\rangle}}
\newcommand{\centraliser}[2]{Z_{S_{#2}}(#1)}
\newcommand{\lattice}[1]{\left(\Z\oplus \Z\right)^{\oplus #1}}
\newcommand{\Zm}[2]{\widetilde{m}_{#1}^{#2}}

\newcommand{\SN}{S_N}
\newcommand{\Sx}[1]{S_{#1}}
\newcommand{\Repq}[1]{\operatorname{Rep}_q(#1)}

\newcommand{\powth}[1]{\operatorname{#1}^{\text{th}}}
\newcommand{\Commutator}[2]{N_{#1,#2}}
\newcommand{\dotp}[2]{{#1} \bullet {#2}}

\DeclareMathOperator{\Hom}{Hom}
\DeclareMathOperator{\End}{End}
\DeclareMathOperator{\cl}{cl}
\DeclareMathOperator{\Id}{Id}
\DeclareMathOperator{\im}{Im}
\DeclareMathOperator{\tors}{tors}
\DeclareMathOperator{\coker}{coker}
\DeclareMathOperator{\tr}{tr}
\DeclareMathOperator{\ord}{ord}

\DeclareMathOperator{\pf}{\mathcal{Z}}

\DeclareMathOperator{\SkCatA}{SkCat}

\tikzset{
  strand/.style={semithick}, 
  dstrand/.style={dotted}, 
  rdstrand/.style={semithick,dash pattern=on 0.7 off 1pt, color=red},  
  oriented/.style={postaction={decorate},
                   decoration={markings, mark=at position 0.85 with {\arrow{>}}}} 
}

\newcommand{\undercross}[2]{
  \vcenter{\hbox{\begin{tikzpicture}[scale=0.6]
    \draw[#1,oriented] (-0.6,-0.6) -- (0.6,0.6);
    \draw[white,line width=5pt] (0.6,-0.6) -- (-0.6,0.6); 
    \draw[#2,oriented] (0.6,-0.6) -- (-0.6,0.6);
  \end{tikzpicture}}}
}

\newcommand{\overcross}[2]{
  \vcenter{\hbox{\begin{tikzpicture}[scale=0.6]
    \draw[#2,oriented] (0.6,-0.6) -- (-0.6,0.6); 
    \draw[white,line width=5pt] (-0.6,-0.6) -- (0.6,0.6); 
    \draw[#1,oriented] (-0.6,-0.6) -- (0.6,0.6); 
  \end{tikzpicture}}}
}

\newcommand{\parallelstrands}[2]{
  \vcenter{\hbox{\begin{tikzpicture}[scale=0.6]
    \draw[#1,oriented] (-0.4,-0.6) -- (-0.4,0.6);
    \draw[#2,oriented] (0.4,-0.6) -- (0.4,0.6);
  \end{tikzpicture}}}
}

\newcommand{\singlestrand}[1]{
  \vcenter{\hbox{\begin{tikzpicture}[scale=0.6]
    \draw[#1,oriented] (0,-0.6) -- (0,0.6);
  \end{tikzpicture}}}
}

\newcommand{\loopstrandJ}[1]{
  \vcenter{\hbox{\begin{tikzpicture}[scale=0.6]
    \coordinate (C) at (-0.1,0);
    \coordinate (X) at (-0.1,-0.6);
    \coordinate (Y) at (-0.1,0.6);
    \coordinate (A0) at (0.1,0.1);
    \coordinate (A1) at (0.1,-0.3);
    \coordinate (B) at (0.3,-0.1);
    \draw[#1] (A1) [out=180,in=-80] to (C);
    \draw[#1,oriented] (C) [out=100,in=-90] to (Y);
    \draw[white,line width=4.5pt] (X) [out=90,in=-100] to (C);
    \draw[white,line width=4.5pt] (C) [out=80,in=90] to (B);
    \draw[white,line width=4.5pt] (B) [out=-90,in=0] to (A1);
    \draw[#1] (X) [out=90,in=-100] to (C);
    \draw[#1] (C) [out=80,in=180] to (A0);
    \draw[#1] (A0) [out=0,in=90] to (B);
    \draw[#1] (B) [out=-90,in=0] to (A1);
  \end{tikzpicture}}}
}

\newcommand{\circlestrand}[2]{
  \vcenter{\hbox{\begin{tikzpicture}[scale=0.6]
    \ifthenelse{\equal{#2}{ccw}}{
      \draw[#1,oriented] (0,0) circle (0.5);
    }{
      \draw[#1,oriented] (0.5,0) arc[start angle=0,end angle=-360,radius=0.5];
    }
  \end{tikzpicture}}}
}

\newcommand{\Nvertex}[1]{
  \vcenter{\hbox{\begin{tikzpicture}[scale=0.6]
    \tikzset{lowarrow/.style={semithick, postaction={decorate},
                               decoration={markings, mark=at position 0.25 with {\arrow{>}}}}}

    \draw[strand,lowarrow] (-1.2,-1.2) -- (0,-0.2);
    \draw[strand,lowarrow] (-0.7,-1.2) -- (0,-0.2);

    \node at (0,-1.2) {\scriptsize $\cdots$};

    \draw[strand,lowarrow] (0.7,-1.2) -- (0,-0.2);
    \draw[strand,lowarrow] (1.2,-1.2) -- (0,-0.2);

    \filldraw (0,-0.2) circle (2pt);

    \draw[rdstrand,oriented] (0,-0.2) -- (0,1.0);
  \end{tikzpicture}}}
}

\newcommand{\dualNvertex}[1]{
  \vcenter{\hbox{\begin{tikzpicture}[scale=0.6]
    \tikzset{lowarrow/.style={semithick, postaction={decorate},
                               decoration={markings, mark=at position 0.25 with {\arrow{<}}}}}

    \draw[rdstrand,oriented] (0,-1.0) -- (0,-0.2);

    \filldraw (0,-0.2) circle (2pt);

    \draw[strand,lowarrow] (-1.2,0.8) -- (0,-0.2);
    \draw[strand,lowarrow] (-0.7,0.8) -- (0,-0.2);

    \node at (0,0.8) {\scriptsize $\cdots$};

    \draw[strand,lowarrow] (0.7,0.8) -- (0,-0.2);
    \draw[strand,lowarrow] (1.2,0.8) -- (0,-0.2);
  \end{tikzpicture}}}
}

\newcommand{\NtoNvertex}[1]{
  \vcenter{\hbox{\begin{tikzpicture}[scale=0.6]
    \tikzset{
      lowarrow/.style={semithick,postaction={decorate},
                        decoration={markings, mark=at position 0.25 with {\arrow{>}}}},
      dashedarrow/.style={semithick,postaction={decorate},
                           decoration={markings, mark=at position 0.8 with {\arrow{>}}}},
      outarrow/.style={semithick,postaction={decorate},
                        decoration={markings, mark=at position 0.25 with {\arrow{<}}}}
    }

    \draw[strand,lowarrow] (-1.2,-1.5) -- (0,-0.7);
    \draw[strand,lowarrow] (-0.7,-1.5) -- (0,-0.7);
    
    \node at (0,-1.5) {\scriptsize $\cdots$};

    \draw[strand,lowarrow] (0.7,-1.5) -- (0,-0.7);
    \draw[strand,lowarrow] (1.2,-1.5) -- (0,-0.7);

    \filldraw (0,-0.7) circle (2pt);

    \draw[rdstrand,dashedarrow] (0,-0.7) -- (0,0.0);

    \filldraw (0,0.0) circle (2pt);

    \draw[strand,outarrow] (-1.2,0.8) -- (0,0.0);
    \draw[strand,outarrow] (-0.7,0.8) -- (0,0.0);

    \node at (0,0.8) {\scriptsize $\cdots$};

    \draw[strand,outarrow] (0.7,0.8) -- (0,0.0);
    \draw[strand,outarrow] (1.2,0.8) -- (0,0.0);
  \end{tikzpicture}}}
}

\newcommand{\Nbox}[1]{
  \vcenter{\hbox{\begin{tikzpicture}[scale=0.6]
  
    \tikzset{lowarrow/.style={semithick, postaction=
                               {decorate},
                               decoration={markings, mark=at position 0.6 with {\arrow{>}}}}}

    \draw[strand,lowarrow] (-1.2,-1.5) -- (-1.2,-0.5);
    \draw[strand,lowarrow] (-0.8,-1.5) -- (-0.8,-0.5);
    \node at (0,-1) {\scriptsize $\cdots$};
    \draw[strand,lowarrow] (0.8,-1.5) -- (0.8,-0.5);
    \draw[strand,lowarrow] (1.2,-1.5) -- (1.2,-0.5);

    \draw[thick,fill=white] (-1.4,-0.5) rectangle (1.4,0.5);
    \node at (0,0) {$#1$};

    \draw[strand,lowarrow] (-1.2,0.5) -- (-1.2,1.5);
    \draw[strand,lowarrow] (-0.8,0.5) -- (-0.8,1.5);
    \node at (0,1) {\scriptsize $\cdots$};
    \draw[strand,lowarrow] (0.8,0.5) -- (0.8,1.5);
    \draw[strand,lowarrow] (1.2,0.5) -- (1.2,1.5);
  \end{tikzpicture}}}
}

\title{The skein partition function of the mapping torus}
\author{Julia Bierent, David Jordan, Matthias Vancraeynest, Monica Vazirani}
\date{\today}
\begin{document}

\maketitle

\begin{abstract}
We compute the dimensions of $\GLN$-skein modules of genus-one mapping tori $T^2\times_\gamma S^1$, for an arbitrary diffeomorphism of $T^2$, and for generic quantum parameter.  These are most cleanly expressed via a generating function over all $N$, which we dub the skein partition function, and for which we compute an explicit Euler product expansion.
\end{abstract}
\tableofcontents

\section{Introduction}

The $G$-skein module of a closed, oriented $3$-manifold $M$, denoted by $\SkMod{G}(M)$ is a vector space formed as the span of framed labelled graphs in $M$, modulo local \emph{skein relations}.  The labels, and the skein relations, are modelled on the category $\Repq{G}$ of finite dimensional modules of the quantum group associated to the reductive algebraic group $G$. It was conjectured by Witten, and then proved in \cite{GJS23}, that the skein module is finite dimensional for generic values of the quantum parameter $q$.

To date, rather few skein module dimensions are known and it is an important question to compute more examples and find general patterns and enumerative formulas.  The case $G= \GL{1}=\mathbb{C}^\times$ is known for any closed 3-manifold thanks to \cite{Prz98}.  In the case $G=\SL{2}$, the product manifolds $M=\Sigma_g\times S^1$  and the mapping tori $M_\gamma=T^2\times_\gamma S^1$ for $\gamma\in\SL{2}(\mathbb{Z})$ are known thanks to \cite{DW21} and \cite{Kinnear24}, respectively.  Moving beyond $\mathbb{C}^\times$ and $\SL{2}$, the only known dimensions were for $G=\SL{N},\GL{N}$ and for $M=T^3$ thanks to \cite{GJV23,GJV24}.  Here we extend to the cases $G=\GLN$ and $M=M_\gamma$. 

This paper accomplishes two goals: the first is to reframe the problem of computing $\GLN$-skein module dimensions at generic $q$ using generating functions, and the second is to compute those generating functions in the case of mapping tori in genus one. 

\subsection{The skein partition function}

In this paper we determine the dimension of the $\GL{N}$-skein module  of the mapping torus $M_\gamma=T^2\times_\gamma S^1$, for generic quantum parameter $q$, for any $\gamma\in\SL{2}(\Z)$ and for all $N\ge 1$. We collect the dimensions over all $N$ together into a single generating function, which we call the \emph{skein partition function}.

\begin{definition}
    The \defterm{skein partition function} of a closed oriented 3-manifold $M$ is the generating function,
    \begin{gather} \label{eq:ZM}
    \pf_M(t) = 1 + \sum_{N=1}^\infty \dim\Sk{N}(M)\cdot t^N,
       \end{gather}
of the $\GLN$-skein modules of $M$ at generic quantum parameter $q$ for all $N\geq 1$.
\end{definition}
Recall that any monic integer power series -- i.e. one of the form $f(t) = 1 + \sum_{k=1}^\infty a_k t^k$, with $a_k\in\Z$ -- admits an Euler product expansion of the form,
\[
f(t) = \prod_{k=1}^\infty(1-t^k)^{-c_k},
\]
for some sequence of integers $c_k\in\mathbb{Z}$, determined uniquely  from the $a_k$ via recursion.   Our main result is a computation of the skein partition of mapping tori via determination of its Euler expansion.  This in particular gives simple formulas for the $\GLN$-skein dimensions of mapping tori for all $N$.

\begin{theorem}\label{main result}
    Let $\gamma \in \SLZ$. The skein partition function $\pf_{M_\gamma}(t)$ of the mapping torus $M_\gamma$ admits an Euler product expansion,
    \begin{align*}
        \pf_{M_\gamma}(t) = \prod_{k=1}^\infty (1-t^k)^{-c_k(\gamma)},
    \end{align*}
    where the $c_k(\gamma)$ are all positive  integers determined by the trace $\tr(\gamma)$.
More precisely, we have:
\begin{equation}
        c_k(\gamma) = \frac{1}{k}\sum_{ d|k}  \phi\left(\frac{k}{d}\right) |\tr(\gamma^d) -2|,\label{eqn:formula-generic}
\end{equation}
in the  case that $\tr(\gamma)\neq\pm 2$, or $\gamma^k \neq \Id_2$.  
The remaining special cases occur when $\gamma$ is a conjugate of the shear $T^m$, its negative $-T^m$, or when $\gamma^k = \Id_2$.  In those cases, we have:

\begin{equation*}
    c_k (\gamma)= \begin{cases}
        1,& \text{if } \gamma^k = \Id_2, \\ 
        |m|k,& \text{if $\gamma$ is conjugate to $T^m$ for $m\neq 0$,}\\
        |m|\frac{k}{2}+1, & \text{if $\gamma$ is conjugate to $-T^m$ for $m\neq 0$, and $k$ is even,}\\
        4, & \text{if $\gamma$ is conjugate to $-T^m$ and $k$ is odd.}       

    \end{cases}
\end{equation*}
\end{theorem}

 We prove Theorem \ref{main result} as Corollary \ref{cor:main} below.
 
\begin{remark} \label{rem:Chebyshev}
    We note that the trace $\tr(\gamma^d)$ can be expressed as a polynomial of degree $d$ in $\tr(\gamma)$,
    \[\tr(\gamma^d)=T_d(\tr(\gamma)),\]
    where $T_d(t)\in\Z[t]$ are the first order normalised Chebyshev polynomials. Hence, the exponents $c_k(\gamma)$ asserted in Formula \eqref{eqn:formula-generic} are easily expressed as functions of $\tr(\gamma)$. To illustrate this, we write out $c_k = c_k(\gamma)$ for small values of $k$, as a polynomial in $x:= \tr(\gamma)$ in the case when $\tr(\gamma) >2$. 
    \begin{align*}
        c_1(x) &= x-2,\\
        c_2(x) &= 1/2x^2 + 1/2x - 3 = \frac12 (x^2+x-6),\\
        c_3(x) &= 1/3x^3 - 1/3x - 2 = \frac13(x^3-x-6),\\
        c_4(x) &= 1/4x^4 - 3/4x^2 + 1/2x - 2 = \frac14(x^4-3x^2+2x-8).
    \end{align*}
    
\end{remark}

We give some examples of the skein partition function and its Euler expansion, respectively, for various $\gamma \in \SLZ$ below.

\begin{example} \label{ex:T}
For the shear $ T=\begin{pmatrix}
    1&1\\
    0&1
\end{pmatrix}$, we have:
\begin{align*}
\pf_{M_T}(t) &= 1+ 3t + 6t^2+ 13t^3 +24 t^4 +48 t^5 + 86t^6 + 160 t^7 + 282 t^8 + \cdots
\\
&= \prod_{k=1}^\infty (1-t^k)^{-k}.    
\end{align*}
\end{example}
\begin{example} \label{ex:-T}
For its negative $\gamma=-T$, we have:
\begin{align*}\pf_{M_{-T}}(t) &=1+ 4t +12t^2+ 32t^3 +77 t^4 +172 t^5 + 366t^6 + 744 t^7 + 1460 t^8 +\cdots\\ 
&= \left(\prod_{j=0}^\infty (1-t^{2j+1})^{-4}\right) \left(\prod_{i=1}^\infty (1-t^{2i})^{(-i-1)}\right).\end{align*} 
\end{example}

\begin{example} \label{ex:-I}
For the order two matrix $\gamma=-\Id$, we have:
\begin{align*}
\pf_{M_{-\Id}}(t) &=1+ 4t + 11t^2+ 28t^3 +63 t^4 +132 t^5 + 264 t^6 + 504 t^7 + 928 t^8 +\cdots\\
&= \left(\prod_{j=0}^\infty (1-t^{2j+1})^{-4}\right) \left(\prod_{i=1}^\infty (1-t^{2i})^{-1}\right).
\end{align*}

\end{example}

\begin{example}  \label{ex:hyperbolic}
For the (arbitrarily chosen) hyperbolic element $\gamma =\begin{pmatrix}2 & 1\\[2pt] 3 & 2\end{pmatrix}$, or indeed any $\gamma \in \SLZ$ with $\tr(\gamma)=4$, we have:
\begin{align*}
    \pf_{M_{\gamma}}(t) &= 1+ 2t + 10t^2+ 36t^3 + 142 t^4 +520 t^5 +1980 t^6 + 7344 t^7 + 27550 t^8 \cdots\\
    &=  {(1-t)^{-2} (1-t^2)^{-7}(1-t^3)^{-18}(1-t^4)^{-52}(1-t^5)^{-146}(1-t^6)^{-463}} \cdots.
\end{align*}
\end{example}
\begin{remark}
    \label{rem:plane partition}
For $\gamma=T$, see Example \ref{ex:T}, the dimension of the $\GL{N}$-skein module equals the number of plane partitions of $N$. Its generating function was first studied by MacMahon in \cite{Mac16}. 
\end{remark}
\begin{remark}\label{rem:teenage mutant ninja turtles}
    As we see, comparing $\gamma$ to $-\gamma$ in Examples \ref{ex:T} and \ref{ex:-T}, $\GLN$ skein modules detect mutant manifolds, even for $N=2$.
\end{remark}

\subsection{Positivity and relation to Donaldson-Thomas invariants}

It is a special property of a monic integer power series when all exponents $c_k$ in its Euler expansion are non-negative.  This property is referred to generally as \emph{positivity}.  The Hilbert series of any symmetric algebra $S(V)$ of a finitely graded vector space $V$ exhibits such positivity, and conversely any monic integer power series exhibiting positivity is equal to the Hilbert series of some such $S(V)$.  Our results therefore situate the skein partition function in this class of generating functions, and prompt the question: \emph{is there an intrinsic construction of such a graded vector space $V$?}

A potential answer may be found by through a deep, as yet conjectural, connection between skein modules and cohomological Donaldson-Thomas (DT) invariants.  Both skein modules and cohomological DT invariants can be viewed as defining deformation quantisations of the representation variety, and of the character stack of $M$, each equipped with a canonical (-1)-shifted symplectic structure.  In \cite{GJS23,GS23}, these two quantisations were conjectured to bear a close relation.  Namely, the cohomological DT invariant defines a perverse sheaf of ``vanishing cycles" on the representation variety, whose (properly normalised, underived) global sections should coincide with the skein module and whose higher global sections should coincide with a yet-to-be-defined ``derived skein module".  This relation has been discussed subsequently in \cite{BDINKP25,Kau24}, however as yet there are no applications, owing in part to the present lack of a self-contained definition of derived skein modules.

Cohomological DT invariants admit two-variable partition functions -- two-variable generating functions in which variable degrees capture the $\GLN$-degree as above and the cohomological degree, respectively.  Positivity for cohomological DT invariants has been established in \cite{davison-mandel, davison2020cohomological,Efimov2012} following \cite{hausel2013positivity}, and there is moreover an intrinsic construction of the graded vector space $V$ as a super-Lie algebra, the so-called BPS cohomology.

Hence the positivity property of the skein partition functions computed in this paper is consistent with a conjectural DT-skein correspondence.  Leveraging the conjectured correspondence further, we may make the following conjecture:

\begin{conjecture}
The skein partition function of any closed oriented 3-manifold admits an Euler expansion with non-negative exponents $c_k$.
\end{conjecture}

We moreover expect that the derived skein partition function exhibits positivity, and carries an algebra structure induced by parabolic induction and restriction domain walls between $\GLN$ skein theories.  We expect an isomorphism between this algebra and the symmetric algebra of the BPS cohomology.

\begin{remark}
It may appear natural to define skein partition function for other series of groups, e.g. for other of the classical series as a generating function encoding the skein module over all ranks.  Although we do expect parabolic induction and restriction to play a role in connecting skein theories in other types, we note that this will \emph{not} imply positivity of the associated generating functions. This stems from the rather more intricate structure of Levi subgroups of classical groups outside the $\GLN$ series, for which the only Levi subgroups are products $\GL{\lambda}:=\prod_k \GL{\lambda_k}$ for compositions $\lambda \vDash N$. For instance, one may easily compute from the formulas in \cite{GJV24} that the generating function encoding $\SLN$-skein module dimensions of the torus $T^3$ does not exhibit positivity.  Indeed, the appearance of the convolution in the dimension formulas in \cite{GJV24} simply reflects the divisor structure which appears in the Levi lattice of $\SLN$.
\end{remark}
\subsection{Outline of the paper}
In Section \ref{sec:preliminaries}  we set up notation and review necessary background, such as the definition of the $\GLN$ skein category, algebra, and module in Section \ref{sec:GLN skein}.  In Section \ref{sec:Sn} we set up our notation for the symmetric group.
In Section \ref{sec:hochschild} 
we explain how we will compute the dimension of the $\GLN$ skein module of the mapping torus, whose definition is given in \ref{sec: mapping torus}, via computing coinvariants of  twisted Hochschild homology, and
how we reduce that to a computation involving linear algebra and finitely generated abelian groups. 

Section \ref{sec:dimension of the skein module} contains the main proofs and computations.  In Section \ref{sec:Cokernel} we identify the  Hochschild homology   twisted  by a permutation as a tensor product over that of its disjoint cycles.
In Section \ref{sec:Coinvariants for k-cycle} we prove our main Theorem \ref{main result} as Corollary \ref{cor:main}. We explain how the action of a centraliser subgroup of $\SN$ interacts with  our previous computational reductions, and we give formulas for the dimension of the coinvariant space of the twisted Hochschild homology via orbit counting. In particular, we reduce our computations to understanding the case of a single $k$-cycle in $\Sx{k}$ -- this reduction to $k$-cycles is ultimately what accounts for the determination of the Euler exponents $\ckg$, and in particular their positivity. 

\subsection{Acknowledgements}
We are grateful to Jennifer Brown and Kevin Walker for helpful conversations throughout this project.  We are also grateful to Ben Davison, Sam Gunningham and Pavel Safronov for explaining aspects of the skeins-to-DT correspondence, which greatly motivated and structured our work.

The work of David Jordan was supported by the Simons Foundation award 888988 as part of the Simons Collaboration on Global Categorical Symmetry, and by the EPSRC Open Fellowship ``Complex Quantum Topology", grant number EP/Y008812/1. The work of Monica Vazirani was partially supported by Simons Foundation Collaboration Grant  707426.

\section{Preliminaries}
\label{sec:preliminaries}
\subsection{\texorpdfstring{$\GL{N}$}{GLN} skein theory} \label{sec:GLN skein}
Throughout the paper, $\K$ denotes the field $\C(q)$ -- equivalently we will always take the parameter $q$ to be generic. Skein theory can be developed for arbitrary ribbon categories. One can recover the following definitions, when considering the ribbon category $\Repq{\GLN}$. For a discussion of the general theory, we refer the interested reader to \cite{Tur94}. First, we introduce the notion of a $\GLN$-skein. 
\begin{definition}
Let $M$ be an oriented 3-manifold.  We define a \defterm{basic $\GLN$-skein} in $M$ to be an oriented, embedded ribbon graph, where the edges are either solid (black) or dashed ({{\color{red} red}}) and where each vertex is one of the following $(N+1)$-valent vertices
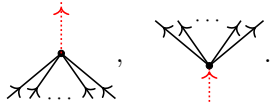
\begin{align*}
    \Nvertex{N},\quad \dualNvertex{N}.
\end{align*}
\end{definition}
When we draw $\GLN$-skeins on the page and without indicating a framing, we assume blackboard framing.  We note that the empty graph is allowed, and will be equivalent modulo skein relations to any graph labelled entirely by the unit object. We will denote by $\varnothing$ the empty skein.  Let us denote the quantum integers and quantum factorials by
\[[N] = \frac{q^N-q^{-N}}{q-q^{-1}},\qquad [N]! = [N]\cdot[N-1]\cdots [1].
\]
\begin{proposition}

Let $M$ be an oriented, compact $3$-manifold. The \defterm{$\GLN$-skein module} $\Sk{N}(M)$, is defined as the  $\K$-span of the isotopy classes of basic $\GLN$-skeins embedded in $M$, modulo the following skein relations

\begin{gather} 
    \overcross{strand}{strand}
  \;-\;
  \undercross{strand}{strand}
  \;=\; \bigl(q-q^{-1}\bigr)\;
  \parallelstrands{strand}{strand}, \qquad  q^{-N}\overcross{rdstrand}{rdstrand}
  \;=\; \parallelstrands{rdstrand}{rdstrand} \;=\;
  q^{N}\undercross{rdstrand}{rdstrand}, \notag \\  \notag
  \circlestrand{strand}{ccw} = [N] \, \varnothing, \qquad \loopstrandJ{strand} = q^N \singlestrand{strand}, \qquad \circlestrand{rdstrand}{ccw} =  \varnothing, \qquad  \loopstrandJ{rdstrand} = q^{N} \singlestrand{rdstrand},\\\overcross{strand}{rdstrand}
   \;=\;
  q^{2}\undercross{strand}{rdstrand}, \qquad \overcross{rdstrand}{strand}
 \;=\;
  q^{2}\undercross{rdstrand}{strand},\label{skein relations}
\\ \notag
  \NtoNvertex{N} = \quad \frac{q^{\binom{N}{2}}}{[N]!} \sum_{\sigma \in \SN} (-q)^{-\ell(\sigma)} \, \Nbox{\sigma}, \notag
\end{gather}
   applied in embedded cylinders $D^2\times I\hookrightarrow M$.  
\end{proposition}

\begin{remark}
The $\GLN$-skein module defined here can easily be shown to be equivalent to: (a) the special case $\mathcal{A}=\Repq{\GLN}$ of the general definition as in \cite{Tur94}, and (b) the evident $\GLN$-modification of $\SLN$-webs from \cite{Sik05}, and $\SLN$-spiders from \cite{CKM}. The solid black line represents the defining representation, and the {{\color{red} red}} dashed line represents its $\powth{N}$ quantum exterior representation, the quantum determinant.
\end{remark}
\begin{definition}
Let $\Sigma$ be an oriented, compact surface. The \defterm{$\GLN$-skein category of $\Sigma$}, denoted $\SkCatA_{\GLN}(\Sigma)$ is the $\K$-linear category with
    \begin{itemize}[label=-]
        \item objects: finite unions of oriented, framed points on $\Sigma$, coloured by either ``black solid" or ``{{\color{red} red }} dashed". This data is also called a labeling of $\Sigma$.
        \item morphisms space: the $\K$-vector space spanned by isotopy classes of $\GLN$-skeins in $M= \Sigma\times I$ compatible with the labeling of the source and target objects, modulo the skein relations \eqref{skein relations} applied in embedded cylinders $D^2\times I\hookrightarrow \Sigma \times I$.  
    \end{itemize}
    The composition is defined by stacking in the $I$-direction.
\end{definition}
\begin{definition}
    The endomorphism algebra $\End_{\SkCatA_{\GLN}(\Sigma)}(\varnothing)$ for the empty labeling $\varnothing$, is called the \defterm{$\GLN$-skein algebra of $\Sigma$} and is denoted by $\SkAlg{N}(\Sigma)$.
\end{definition}
\subsection{Symmetric group conventions}
\label{sec:Sn}
 We recall some standard notions and fix notation regarding the symmetric group $\SN$. For a permutation $w \in \SN$, we write $\lambda = \cycletyperep{N}$ to denote its cycle type, that is $w$ is the product of $\s{k} = \s{k}(w)  = \s{k}(\lambda)$ disjoint cycles of length $k$. 
 If $\sum_k k r_k  = N$ then we say $\lambda $ is a partition of $N$, denoted $\lambda \vdash N$. 
 When $w$ or $\lambda $ is fixed, we merely write $\s{k}$ to lighten notation. As  elements in $\SN$ are conjugate if and only if they have the same cycle type, each conjugacy class is characterised by the cycle type of a representative.
 Thus, we write $[w]$ or $\cycletype{N}$ to denote the conjugacy class with representative $w$, and we denote by $\cl(\SN)$ the set of conjugacy classes of $\SN$. When we need to choose a specific representative $w$ of the class $\cycletype{N}$, we will take for simplicity the natural one. That is, we write $w$ as a product of $r=\sum_j r_j(w)$ disjoint cycles by writing the numbers $1 2 \cdots N$ and inserting $r$ pairs of parenthesis in the obvious way.  See Example \ref{ex:cycletype} below. In \eqref{skein relations}, $\ell(w)$ denotes the Coxeter length of $w$. For $w$ of cycle type $k^{\s{}}$, its centraliser is isomorphic to the wreath product $\centraliser{w}{kr} \cong\cyclicgr{k} \wr S_{\s{}}$, which is defined as the semidirect product $\cyclicgr{k}^{\times \s{}} \rtimes S_{\s{}}$, where 
each copy of the cyclic group $\cyclicgr{k}$ has as generator a single $k$-cycle of $w$, and
 $S_{\s{}}$ acts by permuting these $\s{}$ copies of $\cyclicgr{k}$. For a general element $w$ of cycle type $\cycletyperep{N}$, its centraliser satisfies $\centraliser{w}{N} \cong \prod_{k=1}^{N}\cyclicgr{k} \wr S_{\s{k}}$.
\begin{example}
\label{ex:cycletype}
    The conjugacy class $[1^3 3^2]$ in $S_9$ has $r_1=3, r_3=2$ and $r_j=0$ otherwise. We have not written the  parts $j^0$ when $r_j=0$. This conjugacy class has natural representative given by $w=(1)(2)(3)(4,5,6)(7,8,9)$.    The centraliser  of $w$ in $S_9$ is isomorphic to $S_3 \times (\cyclicgr{3} \wr S_2)$, generated by 
    $\{ \underbrace{(1,2), (2,3)}_{S_3}, \underbrace{(4,5,6)}_{C_3}, \underbrace{(7,8,9)}_{C_3}, \underbrace{(4,7)(5,8)(6,9)}_{S_2}\}$.
\\
    In the notation of Theorem \ref{thm:result general element} below, 
    $w_1 = (1)(2)(3)$ and  $w_3 = (4,5,6)(7,8,9)$.
\end{example}
\subsection{Mapping torus and \texorpdfstring{$\SL{2}(\Z)$}{sl2Z}} \label{sec: mapping torus}
A standard reference for this section is \cite{Hat23}.
\begin{definition}
    The \defterm{mapping torus} of diffeomorphism $\gamma$ of the $2$-torus $T^2$ is defined by
    \begin{equation*}
        M_\gamma := \quotient{T^2\times I}{(x,0)\sim (\gamma(x),1)}.
    \end{equation*}
\end{definition}
Note that up to homeomorphism, $M_\gamma$ depends only on the representative in the mapping class group $\SL{2}(\Z)$, and that moreover $M_\gamma, M_{\gamma'}$, for $\gamma,\gamma' \in \SL{2}(\Z)$ are orientation-preserving diffeomorphic if and only if $\gamma$ and $\gamma'$ are conjugate in $\SLZ$. 

Recall that the group $\SL{2}(\Z)$ can be generated by the matrices  $S=\begin{pmatrix}
    0&-1\\
    1&0
\end{pmatrix}$ and $ T=\begin{pmatrix}
    1&1\\
    0&1
\end{pmatrix}$.
With respect to these generators, $\SLZ$ has corresponding presentation
\begin{equation*}
    \SL{2}(\Z)=\langle S,T\mid S^4 = 1, (ST)^3=S^2\rangle.
\end{equation*}

\begin{definition}
    For $\gamma \in \SL{2}(\Z)$, we say that $\gamma$ is \defterm{hyperbolic} if
    \begin{equation*}
        |\tr(\gamma)|>2.
    \end{equation*}
\end{definition}

If $\gamma$ is not hyperbolic then  it  is conjugate in $\SL{2}(\Z)$ to one of the following:

\begin{alignat*}{5}
&\begin{pmatrix}
    1 & 0 \\
    0 & 1
\end{pmatrix} & = & \, \Id 
&\quad& \ord(\gamma) &=& 1,\;\; \tr(\gamma)=2 \\[6pt]
&\begin{pmatrix}
    -1 & 0 \\
    0 & -1
\end{pmatrix} & = & -\Id 
&\quad& \ord(\gamma) &=& 2,\;\; \tr(\gamma)=-2 \\[6pt]
&\begin{pmatrix}
    -1 & 1 \\
    -1 & 0
\end{pmatrix} & = & -TS \;\text{or}\; 
\begin{pmatrix}
    0 & -1 \\
    1 & -1
\end{pmatrix} = -(TS)^{-1} 
&\quad& \ord(\gamma) &=& 3,\;\; \tr(\gamma)=-1 \\[6pt]
&\begin{pmatrix}
    0 & \mp 1 \\
    \pm 1 & 0
\end{pmatrix} & = & \pm S 
&\quad& \ord(\gamma) &=& 4,\;\; \tr(\gamma)=0 \\[6pt]
&\begin{pmatrix}
    1 & -1 \\
    1 & 0
\end{pmatrix} & = & \, TS \;\text{or}\; 
\begin{pmatrix}
    0 & 1 \\
    -1 & 1
\end{pmatrix} = (TS)^{-1} 
&\quad& \ord(\gamma) &=& 6,\;\; \tr(\gamma)=1 \\[6pt]
&\begin{pmatrix}
    1 & m \\
    0 & 1
\end{pmatrix} & = & \, T^m,\quad m\in\Z\setminus\{0\} 
&\quad& \ord(\gamma) &=& \infty,\;\; \tr(\gamma)=2 \\[6pt]
&\begin{pmatrix}
    -1 & -m \\
    0 & -1
\end{pmatrix} & = & -T^m,\quad m\in\Z\setminus\{0\} 
&\quad& \ord(\gamma) &=& \infty,\;\; \tr(\gamma)=-2.
\end{alignat*}

See Figure \ref{fig:overview table} for an overview of our results for all possible conjugacy classes of $\gamma\in\SL{2}(\Z)$. 
Note that above we have written  $\Id$ for $\Id_2$ and from now on, if not specified, $\Id=\Id_2$.
\subsection{Skein module of the mapping torus via Hochschild homology}
\label{sec:hochschild}
We recall the definition of the Hochschild homology of a $\K$-linear category $\mathcal{C}$. It is given by the coend of the $\Hom$-bifunctor,
\begin{align*}
    \HH(\mathcal{C}) = \int^{x \in \mathcal{C}} \Hom(x,x) = \left( \bigoplus_{x \in \mathcal{C}} \mathrm{Hom}(x, x) \right) 
\Big/ 
\left( f \circ g - g \circ f \,\middle|\, g : x \to y, f : y \to x \text{ for } x, y\in \mathcal{C} \right).
\end{align*}
In the case where $\mathcal{C}$ is the skein category $\SkCat{\mathcal{A}}(\Sigma)$ for  a closed oriented surface $\Sigma$, it was proven in \cite{GJS23} that there exists a canonical isomorphism of vector spaces
\begin{align}
\label{eq:skein module is HH0}
\SkMod{\cA}(\Sigma \times S^1) \cong \HH(\SkCat{\cA}(\Sigma)).
\end{align}
 In \cite{Kinnear24}, this was extended to  mapping tori $\Sigma \times_{\gamma} S^1$. A diffeomorphism of $\Sigma$ acts on the objects and morphisms of the skein category by acting on $\Sigma \times I$ as $\gamma \times \Id$. This allows for a twisted version of Hochschild homology,
\begin{align*}
    \HH^\gamma(\mathcal{C})&\cong 
\int^{x \in \mathcal{C}} \Hom(\gamma(x), x)\\
&= \left( \bigoplus_{x \in \mathcal{C}} \Hom(\gamma(x), x) \right) 
\Big/ \left( f \circ g - g \circ \gamma(f) \ \middle| \ f : y \to x,\, g : \gamma(x) \to y \text{ for }  x,y \in \mathcal{C}\right).
\end{align*}
In a similar manner one obtains the following isomorphism
\begin{equation*}
    \SkMod{\cA}(\Sigma \times_{\gamma} I) \cong \HH^\gamma(\SkCat{\cA}(\Sigma)),
\end{equation*}
which reduces to \eqref{eq:skein module is HH0}, when $\gamma = \Id$.

\subsection{Hochschild homology of the skein category via the  quantum torus}
\label{sec: twisted homology}
The skein algebra $\SkAlg{N}(T^2)$ arises as the endomorphism algebra of the empty skein object of the skein category. Hence, the skein algebra (considered as category with one object) embeds into the skein category $\SkCat{\GL{N}}(T^2)$.  In \cite{GJV24}, it was shown that this inclusion induces a Morita equivalence. Since Hochschild homology is a Morita invariant, one obtains the isomorphism, 
\begin{align*}
    \Sk{N}(M_\gamma) \cong \HH^\gamma(\SkAlg{N}(T^2)).
\end{align*}
In \cite{GJV23}, a further isomorphism was given,
\begin{align}\label{eq:skeinalg is A}
    \SkAlg{N}(T^2) \cong (A_N)^{\SN},
\end{align}
where $A_N$ denotes the algebra, also known as the quantum torus, defined as follows.

Let $\bullet$ denote  the dot product on $\Z^{\oplus N}$.
    Let $\omega$ be the nondegenerate skew pairing on $\lattice{N}$ given by
\[\omega(\mathbf{u},\mathbf{u'})=\dotp{(j_1,\ldots,j_N)}{(\ell_1',\ldots,\ell_N')}-\dotp{(\ell_1,\ldots,\ell_N)}{(j_1',\ldots,j_N')},\]
for $\mathbf{u}=(j_1,\ell_1\ldots,j_N,\ell_N)$, $\mathbf{u'}=(j_1',\ell_1'\ldots,j_N',\ell_N')$ in $\lattice{N}$.
    
\begin{definition}
    \label{def:algebra A}
    Let $A_N=\K_{\omega}[\lattice{N}]$ denote the twisted group algebra, whose underlying vector space is the $\K$-span of symbols ${m_\mathbf{u}}$, for  $\mathbf{u}\in\lattice{N}$,
     and with multiplication
    \begin{align*}
        m_{\mathbf{u}}m_{\mathbf{v}}=q^{\frac{1}{2}\omega(\mathbf{u},\mathbf{v})}m_{\mathbf{u+v}}.
    \end{align*}

  \end{definition}     
    
    The group $S_N$ acts diagonally on $m_{\mathbf{u}}\in A_N$ by permuting the $j_i$'s and $\ell_i$'s simultaneously in $\mathbf{u}=(j_1,\ell_1\ldots,j_N,\ell_N)$, i.e., $w(m_{\mathbf{u}})=m_{w(\mathbf{u})}$ for $w \in S_N$. Note that our skew pairing $\omega$ is $\SN$-invariant. For $w\in S_N$, $\gamma \in \SL{2}(\Z)$  one can define their conjoint action on $A_N$ by multiplication by the tensor product $w\otimes \gamma$ on $\Z^{\oplus N} \otimes_{\Z} \Z^{\oplus 2}  \simeq  \lattice{N}$.

When $N$ is clear from context, to lighten notation we merely write $A$ for $A_N$. 

\begin{remark}\label{rem:N2 vs 2N}
    While one may prefer to consider $\gamma \otimes w$ acting on $\Z^{\oplus 2} \otimes_{\Z} \Z^{\oplus N} $ which in turn is more naturally identified with  $ \Z^{\oplus N} \oplus \Z^{\oplus N}$, the calculations performed in Section \ref{sec:Cokernel} were simpler to execute viewing $w\otimes \gamma$ acting on $\lattice{N}$.  Thus, we also label the generators of $A_N$ accordingly.
\end{remark}
In \cite{GJV24}, and then in \cite{Kinnear24} for the twisted version, the  authors continue their study and employ the isomorphism
\begin{align}
\label{eq-twisted}
    \HH^{\gamma}(A^{\SN}) \cong \HH^{\gamma}(A\#{\SN}) \cong\bigoplus_{[w]\in \cl(\SN)} \HHA{\gamma}{w}_{\centraliser{w}{N}},
\end{align}
where the right hand side is defined as follows.
Given $w\in S_N$ and $\gamma \in \SL{2}(\Z)$, the action of $w\otimes \gamma$ allows us to define the $A$-bimodule $A_{w\gamma}$ having the same underlying space  $A$, on which $A$ acts on the left in the usual way, and on the right twisted by $w\otimes\gamma$. The $\powth{zero}$ Hochschild homology of $A$ with coefficients in the twisted bimodule $A_{w\gamma}$ is then given by 
\begin{equation}
\label{eqn:twisted HH0}
    \HHA{\gamma}{w}=\quotient{A_{w\gamma}}{\Commutator{w}{\gamma}},
\end{equation}
where we quotient by the \emph{subspace of twisted commutators} 
\begin{equation}
\label{eq:commutator relations}
\Commutator{w}{\gamma}:=\K\{m_{\mathbf{u}}m_{\mathbf{v}}-m_{\mathbf{v}}m_{w\otimes\gamma(\mathbf{u})} \mid \mathbf{u}, \mathbf{v} \in \lattice{N}\}.
\end{equation}

 The sum in \eqref{eq-twisted} is over conjugacy classes and the subscript $\centraliser{w}{N} $ indicates that we take coinvariants with respect to the centraliser of $w \in \SN$. Implicit in the statement of the isomorphism is the fact that the summands are independent of the choice of representative $w$.

\begin{notation}
    \label{not:basis of lattice}
    For $\gamma\in \SL{2}(\Z)$, $w\in \SN$, and $\mathbf{v} \in \lattice{N}$, we perform the same rescaling of generators of $\HHA{\gamma}{w}$ as in \cite{GJV24}:
    \begin{equation*}
        \widetilde{m}_{\mathbf{v}}:=q^{\frac{1}{2}\omega((\M{\gamma}{w}{N})^{-1}\mathbf{v},\mathbf{v})}m_{\mathbf{v}},
    \end{equation*}
which is well defined even when $\M{\gamma}{w}{N}$ is not invertible because we quotient by $\Commutator{w}{\gamma}$, see \cite{GJV24} or \cite{Kinnear24} for a proof.
\end{notation}

After rescaling, the commutator relations \eqref{eq:commutator relations} in $\HHA{\gamma}{w}$ are proportional to $\widetilde{m}_{\mathbf{u}}-\widetilde{m}_{\mathbf{u}+(w \otimes \gamma)\mathbf{v}}$,
hence as vector spaces we have the isomorphism
\begin{equation}
    \label{eq:dimension HH0 theory}
    \HHA{\gamma}{w} \cong \K[\coker(\M{\gamma}{w}{N})_{\tors}],
\end{equation}
where on the right hand side we have the group algebra of the torsion subgroup of the finitely generated  abelian group $\coker(\Id_{2N}-w\otimes \gamma)$.
\section{Dimension of the skein module}
\label{sec:dimension of the skein module}
In this section, we  compute the dimension of $\HHA{\gamma}{w}_{\centraliser{w}{N}}$ for $w\in \SN$. In order to do so, we will first use equation \eqref{eq:dimension HH0 theory} and compute the image of $\Mwg$ in $\lattice{N}$, and consequently its cokernel. 
In  the case where $w$ is a single $k$-cycle, we reduce the problem to identifying this cokernel as a quotient of $\Z^{\oplus 2}$ in Lemma \ref{lem:coker-equals-coker}.

Let us begin with a brief overview of the logical flow of this section.  Since any $\sigma\in\centraliser{w}{N}$ commutes with $\Mwg$, the quotient
 $\coker(\M{\gamma}{w}{N})_{\tors}$ carries an action of $\centraliser{w}{N}$. Furthermore, 
since the skew paring $\omega$ is $\SN$-invariant 
 we have that for all $\sigma\in\centraliser{w}{N}$, $\sigma(\widetilde{m}_{\mathbf{u}})=\widetilde{m}_{\sigma(\mathbf{u})}$ 
 in $\HHA{w}{\gamma}$. 
 This allows us in Theorem \ref{thm:result general element} to reduce the computation of coinvariants to one of orbit-counting, which becomes combinatorial.

In order to count orbits, we first study the action of the centraliser of $w$ on $\coker(\Mwg)_{\tors}$ in the case that $w$ is a single $k$-cycle. This reduces to understanding $\coker(\Id_2 - {\gamma}^{k})_{\tors}$ with $\gamma$ acting by multiplication. 
Finally, by counting the number of orbits of the action by $\centraliser{w}{N}$ for any $w$, we  give an explicit formula for the dimension of $\HHA{\gamma}{w}_{\centraliser{w}{N}}$ that depends only on  the cycle type of $w$ and the trace of $\gamma$ (see Remark \ref{rem:Chebyshev}).

\subsection{Twisted Hochschild homology expressed as a cokernel}
\label{sec:Cokernel}
The first step is to compute the image of $\M{\gamma}{w}{N}$ via column reduction.  We find we need only compute the image of the $2 \times 2$ matrices $\Id-\gamma^k$ for appropriate $k$. 
\begin{lemma}\label{lem:coker-equals-coker}
    Let $\gamma \in \SL{2}(\Z)$ and $w\in S_N$ have cycle type $\cycletyperep{N}$, then we have an isomorphism of abelian groups,
    \begin{equation*}
       \coker(\M{\gamma}{w}{N})
       \cong \bigoplus_{k=1}^N(\coker(\Mred{\gamma}{k}))^{\oplus \rkw}.
    \end{equation*}
    Consequently combining equation \eqref{eq:dimension HH0 theory} with taking torsion and linearizing, we have:
    \begin{equation*}
      \HHA{\gamma}{w} \cong \K[\coker(\M{\gamma}{w}{N})_{\tors}]
       \cong \K[\bigoplus_{k=1}^N(\coker(\Mred{\gamma}{k})_{\tors})^{\oplus \rkw}].
    \end{equation*}
\end{lemma}
\begin{proof}
We consider the cokernel of the $2N\times 2N$ matrix $\M{\gamma}{w}{N}$, with $w\in \SN$. Since only the conjugacy class of $w$ is relevant, we may choose the natural representative $w$ as discussed in Section \ref{sec:Sn}. Then the matrix $\M{\gamma}{w}{N}$ is a diagonal block matrix,
 \begin{align*}
    \M{\gamma}{w}{N} = \mathrm{ diag}(Q_{1,\s{1}},Q_{2,\s{2}}, \dots, Q_{n,\s{n}}),
\end{align*}
consisting of blocks $Q_{k,\s{k}}$ defined by
\begin{align}\label{Q matrix}
    Q_{k,\s{k}} := \mathrm{ diag}(\underbrace{P_k,P_k, \dots, P_k}_{\s{k}\text{ times}}), \quad
    P_k:=\M{\gamma}{\cycle{k}}{k}, \quad 1\leq k \leq n,
\end{align}
and $   P_1= \Id-\gamma$.
Hence, we have isomorphisms 
\begin{equation}
    \label{eq:direct sum coker}
    \coker(\M{\gamma}{w}{N})_{\tors} \cong \bigoplus_{k=1}^N \coker(Q_{k,\s{k}})_{\tors} \cong \bigoplus_{k=1}^N \coker(P_k)^{\oplus \s{k}}_{\tors}.
\end{equation}
We note that, for $1 < k \leq N,$ 
\begin{align}
    \label{eq: reduction of Pk}P_k=\begin{pmatrix}
        \Id & & & &-\gamma\\
        -\gamma &\ddots & & &\\
        &\ddots &\ddots & &\\
         & &\ddots &\Id &\\
         & & &-\gamma &\Id
    \end{pmatrix}&\colred\begin{pmatrix}
        \Id & & & &\\
        -\gamma &\ddots & & &\\
        &\ddots &\ddots & &\\
         & &\ddots &\Id &\\
         & & &-\gamma &\Mred{\gamma}{k}
    \end{pmatrix}\\
    &\rowred \begin{pmatrix}\nonumber
        \Id & & & &\\
         &\Id & & &\\
         & &\ddots & &\\
         & & &\Id &\\
         & & & &\Mred{\gamma}{k}
    \end{pmatrix},
\end{align}
where the first and second equivalences are obtained, respectively, by elementary column and row operations. Hence, there are isomorphisms $\coker(P_k)\cong \coker(\Mred{\gamma}{k})$ and $\coker(P_k)_{\tors}\cong \coker(\Mred{\gamma}{k})_{\tors}$, with $P_k \in \End((\Z \oplus \Z)^{\oplus k}),$  $ \Id-\gamma^k \in \End(\Z \oplus \Z)$. The lemma follows. 
\end{proof}

\subsection{Computing coinvariants of the twisted Hochschild homology via orbit counting}
\label{sec:Coinvariants for k-cycle}

We now study the action of the centraliser $\centraliser{w}{N}$ on $\K[\coker(\M{\gamma}{w}{N})_{\tors}]$. Calculating the dimension of the space of coinvariants with respect to this action amounts to counting the number of orbits of this action on the finite \emph{set} $\coker(\M{\gamma}{w}{N})_{\tors}$ (which is also a finite abelian group). In Theorem \ref{thm:result general element}, we prove that the general case can be reduced to the case where $w=\cycle{k}$ is a single cycle. Then, we will calculate the dimension for one $k$-cycle, i.e., when $\rjw=\delta_{j,k}$, and conclude the result for general $w\in \SN$.
\begin{notation} \label{not:ck}
    We define the positive integers $c_k(\gamma)$ to be  the number of orbits of $\centraliser{\cycle{k}}{k}$ acting on $\coker(\M{\gamma}{\cycle{k}}{k})_{\tors}$. 
\end{notation}
When $\gamma$ is fixed, to lighten notation we will often write $c_k$ for $\ckg$. 
\begin{theorem}\label{thm:result general element}
Let $w$ be an element of $\SN$ of cycle type $\cycletyperep{N}$ and $\gamma \in \SL{2}(\mathbb{Z})$. Then
\begin{align*}
\dim(\HHA{\gamma}{w}_{\centraliser{w}{N}}) = \prod_{k=1}^N\binom{\ckg + \rkw-1}{\rkw}.\end{align*} 
\end{theorem}

Theorem \ref{thm:result general element} immediately implies the form of the Euler expansion in Theorem \ref{main result}, which we recall in the following corollary.

\begin{corollary} \label{cor:main}
We have an equality,
\begin{align*}
        \pf_{M_\gamma}(t) = \prod_{k=1}^\infty (1-t^k)^{-c_k(\gamma)}.
    \end{align*}    
\end{corollary}

\begin{proof}

By Theorem \ref{thm:result general element} we may write the coefficients $a_N$ of $\pf_{M_\gamma}(t)$ as 
\begin{align*}
    a_N 
    =
    \sum_{[w] \in \cl(\SN)} \prod_{k=1}^{N} \binom{c_k(\gamma) + \s{k}(w) -1}{\s{k}(w)}.
\end{align*}

On the other hand, we may directly compute,
\begin{align*}
\prod_{k=1}^\infty (1-t^k)^{-c_k(\gamma)} &=
\prod_{k=1}^\infty \left(1+ \binom{c_k(\gamma)}{1}t^k + \binom{c_k(\gamma) +1}{2}t^{2k}+ \ldots \right)\\
&= 1+ \sum_{N=1}^\infty \sum_{\lambda \vdash N} \prod_{k=1}^{N} \binom{c_k(\gamma) + \s{k}(\lambda) -1}{\s{k}(\lambda)} \,t^N
\end{align*}
Recalling the bijection between partitions of $N$ and conjugacy classes of the symmetric group $\SN$, the formula \ref{eq:ZM} follows. 
\end{proof}

Before giving the proof of Theorem \ref{thm:result general element}, let us state the second main result of this section, Theorem \ref{thm:formula-for-ck} which gives a formula for the $\ckg$. The proof of the latter will take up the remainder of this section.

\begin{theorem}\label{thm:formula-for-ck}
    Suppose that  $\tr(\gamma)\neq\pm 2$, and $\gamma^k \neq \Id$.  Then the following formula holds:
\begin{equation}
\label{eq:ck}
        \ckg = \frac{1}{k}\sum_{ d|k}  \phi\left(\frac{k}{d}\right) |\tr(\gamma^d) -2|.
\end{equation}
The remaining special cases occur when $\gamma$ is conjugate to  a shear $T^m$, its negative $-T^m$, or when $\gamma$ is of finite-order with $\gamma^k=\Id$. 
In those cases, we have the following expressions,
    \begin{equation*}
    \ckg = \begin{cases}
        1,& \text{if } \gamma^k = \Id,\\
        |m|k,& \text{if $\gamma$ is conjugate to $T^m$ for $m\neq 0$,}\\
        |m|\frac{k}{2}+1, & \text{if $\gamma$ is conjugate to $-T^m$ for $m\neq 0$ , and $k$ is even,}\\
        4, & \text{if $\gamma$ is conjugate to $-T^m$ and $k$ is odd.}       

    \end{cases}
\end{equation*}
\end{theorem}

For completeness, we list the value of $c_k(\gamma)$ for all cases in Figure \ref{fig:overview table}.
\begingroup
\renewcommand{\arraystretch}{2}
\begin{figure}[h]
\begin{tabular}{|c| c|c|l|l|} 
 \hline
 $[\gamma]$ & $\tr(\gamma)$ &$\ord(\gamma)$& $\dim\HHA{\gamma}{\cycle{k}}$ & \makecell{$c_k$} \\ [0.5ex] 
 \hline\hline
 \makecell{$S$ or $S^3$} &$ 0$ &$4$ & \makecell[l]{$1 $ if $k \equiv 0 \mod 4$  \\  $2$ if $ k\equiv \pm 1 \mod 4$\\ $4 $ if $ k\equiv 2 \mod 4$} & \makecell[l]{$1 $ if $k \equiv 0 \mod 4$\\ $2$ if $ k\equiv \pm 1 \mod 4$ \\ $3 $ if $ k\equiv 2 \mod 4$} \\
 \hline
 \makecell{$TS$ or $(TS)^{-1}$} & $1$ & $6$&\makecell[l]{$ 1 $ if $k \equiv 0,\pm1 \mod 6$  \\ $ 4 $ if $ k \equiv 3 \mod 6$  \\ $ 3$ if $ k \equiv \pm 2 \mod 6$ }  &\makecell[l]{$ 1 $ if $k \equiv 0,\pm1 \mod 6$  \\ $ 2 $ if $ k \equiv \pm2,3 \mod 6$ }\\
 \hline
 \makecell{$-TS$ or $-(TS)^{-1}$} & $-1 $&$3$& \makecell[l]{$ 1 $ if $ k\equiv 0 \mod 3$  \\ $ 3$ if $ k\equiv \pm 1 \mod 3$} & \makecell[l]{$1 $ if $ k\equiv 0 \mod 3$  \\ $ 3$ if $ k\equiv \pm 1 \mod 3$}\\
 \hline
  \makecell{$\Id$} & $2 $& $ 1$ &$1$ & $1$ \\
 \hline
   \makecell{$-\Id$ }& $-2$ &$2$& \makecell[l]{$1$ if $k \equiv 0 \mod 2$ \\$4$ if $k \equiv 1 \mod 2$ } & \makecell[l]{$1$ if $k \equiv 0 \mod 2$ \\ $4$ if $k \equiv 1 \mod 2$} \\
 \hline
 \makecell{$T^m, \, m \neq 0$} &$ 2$ &$\infty$&  $|m|k$ & $|m|k$\\ 
 \hline
  \makecell{$-T^m, \,  m\neq 0$} & $-2$ &$\infty$ & \makecell[l]{$|m|k$ if $k\equiv 0 \mod 2$ \\$4$ \qquad if $k\equiv 1 \mod 2$ }& \makecell[l]{ $|m|\frac{k}{2}+1$ if $k\equiv 0 \mod 2$ \\ $4$ \qquad if $k\equiv 1 \mod 2$ }\\
 \hline
 hyperbolic  & $|\tr(\gamma)|>2$ &$\infty$&  $|\tr(\gamma^{k})-2|$ &$\frac{1}{k}\sum_{ d|k}  \phi\left(\frac{k}{d}\right) |\tr(\gamma^d) -2|$\\  
 \hline
\end{tabular}
\caption{Overview of the results for all possible conjugacy classes of $\gamma\in\SL{2}(\Z)$.}
\label{fig:overview table}
\end{figure}
\endgroup
\begin{proof}[Proof of Theorem \ref{thm:result general element}]
Since the space of coinvariants is defined by
\[\HHA{\gamma}{w}_{\centraliser{w}{N}}=\quotient{\HHA{\gamma}{w}}{\K\{ \widetilde{m}_{\mathbf{v}}-\widetilde{m}_{\sigma(\mathbf{v})}\mid \sigma \in \centraliser{w}{N},  \mathbf{v} \in \lattice{N}\}}\]
its dimension is equal to the number of orbits of the action of $\centraliser{w}{N}$ on $\coker(\M{\gamma}{w}{N})_{\tors}$.
 Let $w$ be the natural representative of $\cycletype{N}$ and $\gamma \in \SL{2}(\mathbb{Z})$ be fixed. Let $w_k$ denote the product of all the disjoint cycles of $w$ of length $k$, viewed as element of $ S_{k\s{k}}$. The centraliser of $w$ in $S_{n}$ satisfies  $\centraliser{w}{N} \cong \prod_{k=1}^N \centraliser{w_k}{k\s{k}}$. Furthermore, as in Section \ref{sec:Sn}, we have the isomorphism $\centraliser{w_k}{k\s{k}}\cong\cyclicgr{k} \wr S_{\s{k}}$. Recall the isomorphism \eqref{eq:direct sum coker} from the proof of Lemma \ref{lem:coker-equals-coker}, given by $\coker(\M{\gamma}{w}{N}) \cong \bigoplus_{k=1}^N \coker(Q_{k,\s{k}})$.  Note that
the centraliser $\centraliser{w}{N}$ only permutes entries from the cycles of $w$ that have the same lengths. Therefore, we can write
\begin{align*}
    \K[\coker(\M{\gamma}{w}{N})_{\tors}]_{\centraliser{w}{N}}   \cong \bigotimes_{k=1}^N \K[\coker(Q_{k,\s{k}})_{\tors}]_{\centraliser{w_k}{k\s{k}}}.
\end{align*} Alternatively, the set of orbits, denoted $\quotient{\coker(\M{\gamma}{w}{N})_{\tors}}{\centraliser{w}{N}}$, is in bijection with the Cartesian product $\prod_{k=1}^N\left(\quotient{\coker(Q_{k,\s{k}})_{\tors}}{\centraliser{w_k}{k\s{k}}}\right)$. 
We denote by $\quotient{\coker(P_k)_{\tors}}{\cyclicgr{k}}$ the \emph{set} of the $\cyclicgr{k}$-orbits on a single copy of $\coker(P_k)_{\tors}$ and recall from \ref{not:ck} we have denoted by $c_k$ the number of such orbits.  
We have a canonical bijection of sets,
\[
\quotient{\left(\coker(P_k)_{\tors}\right)^{\oplus \s{k}}}{\cyclicgr{k} \wr S_{\s{k}}} \cong \text{Sym}^{\s{k}}\left(\quotient{\coker(P_k)_{\tors}}{\cyclicgr{k}}\right)
\]
between the set of $\cyclicgr{k} \wr S_{\s{k}}$-orbits in $\left(\coker(P_k)_{\tors}\right)^{\oplus \s{k}}$ and the $\powth{r_k}$ symmetric product  of the set of $\cyclicgr{k}$-orbits in $\coker(P_k)_{\tors}$.
Thus, after a standard counting argument, one finds that the number of orbits is equal to $\binom{c_k + \s{k}-1}{\s{k}}$. Hence, the theorem follows.
\end{proof}
As a result of Theorem \ref{thm:result general element}, it suffices to study the action of $\centraliser{\cycle{k}}{k} = \langle \cycle{k}\rangle \cong \cyclicgr{k}$ on $\coker(P_k)_{\tors}$. We  carry out this analysis below.
Consider the case of one $k$-cycle, $w=\cycle{k} \in S_k$. We want to compute the free and torsion parts of $\coker(\Mred{\gamma}{k})$. If $\det(\Mred{\gamma}{k})\neq 0$, the free part is trivial, and we will be able to compute a formula for $\ckg$ using Burnside's formula to count orbits. For $\det(\Mred{\gamma}{k})= 0$, the free part is not trivial and we will discuss it case by case.

\begin{remark}
    \label{rmk:det(Id-gamma^k)}
    For $\gamma\in \SL{2}(\Z)$,
    $
        \det(\Mred{\gamma}{})=0
    $
    if and only if $\tr(\gamma)=2$.
    In particular, in this case
         $\gamma$ is conjugate to $T^\ell$ for some $\ell\in \Z$.  
         Consequently, for $k \in \Z_{>0}$, either
    \begin{enumerate}[label=\alph*)]
        \item \label{case:det non zero} $\det(\Mred{\gamma}{k})\neq 0$, \emph{or},
        \item \label{case:finite order dividing k} $\gamma$ is 
        of finite order dividing $k$, i.e.
        $\gamma^{k}=\Id$, \emph{or},
        \item \label{case:Tmk}$\gamma$ is 
        conjugate to $\pm T^m$, for some $m\in \Z$, hence $\gamma^k$ is conjugate to $ T^{mk}$.
    \end{enumerate}
    
\end{remark}
 In all the cases of Remark \ref{rmk:det(Id-gamma^k)}, we want to understand how the action of the centraliser $\centraliser{\cycle{k}}{k}$ on $\HHA{\gamma}{\cycle{k}}$ descends to an action on $\coker(\Mred{\gamma}{k})_{\tors}$. In the following lemma, taking  $w=\cycle{k}$, we regard $\K[\coker(\Mred{\gamma}{k})_{\tors}]$ as a module for $\centraliser{w}{k}$   via the surjective group homomorphism given by $w \mapsto \bar \gamma $
  where we have written $\bar \gamma$ for its coset in  $ \cycgp{\gamma}{k}$ and where $\gamma$, and hence $\bar \gamma$, acts on $\coker(\Mred{\gamma}{k})_{\tors}$ by multiplication.
\begin{lemma}\label{lem:action of centraliser}
For $\gamma \in \SL{2}(\Z)$, $w=\cycle{k} \in S_k$, the natural action of $\langle \gamma \rangle$ on $\coker(\Mred{\gamma}{k})$ yields that the isomorphism 
\[ \HHA{\gamma}{w} \cong \K[\coker(\Mred{\gamma}{k})_{\tors}],\]
of Lemma \ref{lem:coker-equals-coker} intertwines the natural $\centraliser{w}{k}$ action on the left hand side with the $\centraliser{w}{k}$ action on the right hand side as constructed above.
\end{lemma}
\begin{proof}
    Starting with Notation \ref{not:basis of lattice}, for $\mathbf{v}=(j,\ell)\in \Z\oplus \Z$, we let  
    $\Zm{i}{\mathbf{v}}=\Zm{i}{(j,\ell)}:=\widetilde{m}_{\mathbf{v}_i}$, with $\mathbf{v}_i\in \lattice{k}$ having all components being $0$ except in the $i^{th}$ component, where it is equal to $\mathbf{v}$.
    $$ \mathbf{v}_i = (0,0, \dots, \underbrace{\mathbf{v}}_i,\dots, 0,0) = (0,0, \dots, 0,0, \underbrace{j,\ell}_i,0,0, \dots, 0,0).$$
    From the reduction (\ref{eq: reduction of Pk}), we know that $\HHA{\gamma}{w}$ is the $\K$-vector space spanned by monomials $\Zm{k}{\mathbf{v}}$, with $k$ as above, and $\mathbf{v}$ ranging over $\Z \oplus \Z$, quotiented by the subspace $\K\{ \Zm{k}{\mathbf{v}}-\Zm{k}{\gamma^k\mathbf{v}} \mid \mathbf{v} \in \Z \oplus \Z\}$.
    Now for $\cycle{k}^{k-r}\in \centraliser{\cycle{k}}{k} = \langle \cycle{k}\rangle\subset S_k$, with $0\leq r< k$, we obtain
    \begin{equation*}
        \cycle{k}^{k-r}(\Zm{k}{\mathbf{v}})=\Zm{k-r}{\mathbf{v}}.
    \end{equation*}
    Since we already had, from the cokernel relations \eqref{eq:commutator relations}, that
    \begin{equation*}
        \Zm{k-r}{\mathbf{v}}=\Zm{k-r+1}{\gamma\mathbf{v}}=\cdots=\Zm{k}{\gamma^{r}\mathbf{v}},
    \end{equation*}
    we now have that the action of the centraliser agrees with
    \begin{equation}\label{action of centraliser}
        \cycle{k}^{k-r}
        (\Zm{k}{\mathbf{v}})=\Zm{k}{\gamma^r\mathbf{v}}, \quad 0\leq r< k,
    \end{equation}
    and the lemma follows.
\end{proof}

Let us first discuss the case \ref{case:det non zero} of Remark \ref{rmk:det(Id-gamma^k)}, where $\det(\Mred{\gamma}{k})\neq 0$.
\begin{lemma}
    \label{lem:Number of fixed points}
    If $\det(\Mred{\gamma}{k})\neq 0$, then for $1\leq p <k$, the number of fixed points of $\gamma^p$ on $\coker(\Mred{\gamma}{k})_{\tors}$ is
    \begin{equation*}
        \#\coker(\Mred{\gamma}{k})_{\tors}^{\gamma^p}=|\det(\Mred{\gamma}{d})|.
    \end{equation*}
  with $d=\gcd(k,p)$ or $d = \gcd(k,p,m)$ in the case $\gamma$ has finite order $m$.
\end{lemma}
\begin{proof}
    We consider $\gamma$ of infinite order, respectively finite order $m$.  Let $\bar \gamma^p$ denote the elements of the finite cyclic group $\cycgp{\gamma}{k} = \langle \bar \gamma \rangle$.    Note the order of $ \bar \gamma $ is $k$, respectively $\gcd(k,m)$.  
    We have an equality of the cyclic subgroups $\langle \bar \gamma^p \rangle = \langle \bar \gamma^d \rangle$ exactly when $\gcd(k,p) = \gcd(k, d)$, (respectively  $\gcd(k,p,m) = \gcd(k, d,m)$); and note that any fixed point of $\bar \gamma^p$ will also be a fixed point of $\bar \gamma^{rp}$ for $r \in \Z$.
    Thus we may replace $p$ with $d=\gcd(k,p)$ (respectively  with $d=\gcd(k,p,m)$), which divides $k$.
    When $\det(\Mred{\gamma}{k})\neq 0$, $\coker(\Mred{\gamma}{k})=\coker(\Mred{\gamma}{k})_{\tors}$ and $\Mred{\gamma}{k}$ and $\Mred{\gamma}{d}$  are both endomorphisms of $\Z\oplus\Z=:\Z^2$ with finite cokernel, i.e. whose images are of finite index. 
    Let $L=\im(\Mred{\gamma}{d})\subseteq \Z^2$ and $M=\im(\Mred{\gamma}{k})\subseteq \Z^2$. Because $d$ divides $k$  we have an inclusion of lattices $M\subseteq L\subseteq \Z^2$.  Observe $\coker(\Mred{\gamma}{k})= \mathbb{Z}^2/M$.  We compute the indexes 
    \begin{align*}
        &[\Z^2:L]=|\det(\Mred{\gamma}{d})|,\\
        &[\Z^2:M]=|\det(\Mred{\gamma}{k})|=[\Z^2:L][L:M].
    \end{align*}
    Multiplication by $\Mred{\gamma}{d}$ gives an endomorphism of $\mathbb{Z}^2/M$ with image $L/M$, and with kernel equal to fixed points of $\gamma^d$ acting on $\coker(\Mred{\gamma}{k})$. Hence its kernel has cardinality equal the ratio of the cardinalities of the domain and image, which is $|\det(\Mred{\gamma}{d})|$. This shows that the number of fixed points for $\gamma^d$, and in general for $\gamma^p$ with $\gcd(k,p)=d$, (respectively $\gcd(k,p,m)=d$) is $|\det(\Mred{\gamma}{d})|$.
\end{proof}

\begin{lemma}
    \label{lem:dim HH0 for gamma hyperbolic, k cycle}
    If $\det(\Mred{\gamma}{k})\neq 0$, then
    \begin{equation*}
            \ckg=\frac{1}{k}\sum_{d|k}\phi\left(\frac{k}{d}\right)|\tr(\gamma^d)-2|,
    \end{equation*}
    where $\phi$ is the Euler's totient function.
\end{lemma}
\begin{proof}
By Burnside's Lemma,  the average number of fixed points  of $\cycgp{\gamma}{k}$ acting on $\coker(\Mred{\gamma}{k})_{\tors}$ is given by:
    \begin{gather*}
        \frac{1}{k}\sum_{p=1}^k \#\coker(\Mred{\gamma}{k})_{\tors}^{\gamma^p}
        =\frac{1}{k}\sum_{d|k}\phi\left(\frac{k}{d}\right)|\det(\Mred{\gamma}{d})|.
    \end{gather*}
    The equation follows from Lemma \ref{lem:Number of fixed points} and the definition of $\phi$. 
    We note that the formula holds even in the case that $\gamma$ has finite order $m$, noting our hypothesis precludes $m \mid k$. In this case it still counts the average number of fixed points of a group of size $\gcd(k,m)$ repeated $k/\gcd(k,m)$ times then normalised appropriately. Note that for $\gamma\in\SL{2}(\Z)$, $\det(\Mred{\gamma}{d})=2-\tr(\gamma^d)$, and the lemma follows.
\end{proof}

We are now ready to prove Theorem \ref{thm:formula-for-ck}.

\begin{proof}[Proof of Theorem \ref{thm:formula-for-ck}]
Note that when $\tr(\gamma) \neq \pm 2$, the matrix $\gamma$ is either hyperbolic or of finite order. As mentioned in Remark \ref{rmk:det(Id-gamma^k)}, when $\gamma$ is hyperbolic or of finite order not dividing $k$, it follows that $\det(\Mred{\gamma}{k}) \neq 0$. Hence, Lemma \ref{lem:dim HH0 for gamma hyperbolic, k cycle} applies.\\
\\
When $\gamma$ is of finite order dividing $k$, the matrix $\Mred{\gamma}{k}$ is equal to the zero matrix. Hence,  $\coker(\Mred{\gamma}{k})$ is free, $\coker(\Mred{\gamma}{k})_{\tors}$ is trivial and accordingly, there is only one orbit when acting by the centraliser of $w$. 
In conclusion, we obtain
\begin{align*}
\dim \HHA {\gamma}{\cycle{k}} = \dim \HHA{\gamma}{\cycle{k}}_{\centraliser{\cycle{k}}{k}}=1.
\end{align*}
When $\gamma$ is conjugate to either $T^m$ or $-T^m$ for $m \neq 0$, the computation is slightly more involved. Recall that for conjugate elements of the mapping class groups, their corresponding mapping tori are diffeomorphic. Hence, we can consider, without loss of generality, the cases where $\gamma$ is equal to either $T^m$ or $-T^m$. In the following, we use the notation from the proof of Lemma \ref{lem:action of centraliser}.\\
When $\gamma$ is equal to $T^m$, we have $\coker(\Mred{\gamma}{k})_{\tors}  \cong \{\Zm{k}{(i,0)}\mid 0\leq i < |m|k\}$ as sets. 
Each $ \Zm{k}{(i,0)}$ is fixed by $T$, hence by
equation $\eqref{action of centraliser}$  the action of the centraliser is trivial on $\HHA{\gamma}{\cycle{k}}$.  Hence, the number of orbits is the same as the number of elements. 
In other words, for each $\mathbf{v} \in \lattice{k}$,  $\Zm{\mathbf{v}}{} - \Zm{\cycle{k} \cdot \mathbf{v}}{}$ is in the span of the relations \eqref{eq:commutator relations}.
In conclusion,
\begin{align*}
\ck(T^m)= \dim \HHA{T^m}{\cycle{k}}_{\centraliser{\cycle{k}}{k}}=\dim \HHA{T^m}{\cycle{k}} = |m|k.
\end{align*}
When $\gamma$ is equal to $-T^m$, one needs to distinguish between the case where $k$ is odd or even.
When $k$ is odd, we have $\det(\Mred{\gamma}{k})= 4$, hence Lemma \ref{lem:dim HH0 for gamma hyperbolic, k cycle} applies. 

When $k$ is even, we have $\coker(\Mred{\gamma}{k})_{\tors} \cong \{\Zm{k}{(i,0)}\mid 0\leq i < |m|k\}$, as in the previous case. However, now the action by the centraliser $\centraliser{\cycle{k}}{k} = \langle (12 \cdots k) \rangle$ is non-trivial. Each $-T^m$ will send $\Zm{k}{(i,0)}$ to $\Zm{k}{(- i,0)}$, which is identified with $\Zm{k}{(|m|k-i,0)}$ modulo the relations \eqref{eq:commutator relations}.
 Hence, the centraliser has two singleton orbits  $\{\Zm{k}{(0,0)}\}$ and $\{\Zm{k}{(|m|\frac{k}{2},0)}\}$ and $|m|\frac{k}{2}-1$ orbits of the form $\{\Zm{k}{(i,0)},\Zm{k}{(|m|k-i,0)}\}$. In conclusion, for $k$ even, we obtain
\begin{align*}
\ck(-T^m) = \dim \HHA{(-T^m)}{\cycle{k}}_{\centraliser{\cycle{k}}{k}}=|m|\frac{k}{2}+1.
\end{align*}
\end{proof}
\begin{example} \label{ex:table of ck for various gamma}
    The skein module dimensions for $\GL{N}$ and the sequences $c_k$ are given below, for some representative matrices $\gamma$.
    See Examples \ref{ex:T}, \ref{ex:-T}, \ref{ex:-I}, \ref{ex:hyperbolic}, where we have packaged the data into skein partition functions.

\begin{table}[ht]
\centering \small
\begin{tabular}{c|l|l}
$\gamma$ & $c_k(\gamma)$ for $k = 1,2,3,\dots$ &  $\dim\Sk{N}(M_\gamma)$ for $N =1,2,3, \ldots$ \\[2pt]
\hline
$\begin{pmatrix}2 & 1\\[2pt] 3 & 2\end{pmatrix}$ & 
2,\;7,\;18,\;52,\;146,\;463,\;1442,\;4732,\;15618,\;\ldots & 
2,\;10,\;36,\;142,\;520,\;1980,\;7344,\;27550,\;102686,\;\ldots \\[4pt]

$\Id$ & 
1,\;1,\;1,\;1,\;1,\;1,\;1,\;1,\;1,\;1,\;1,\;\ldots & 
1,\;2,\;3,\;5,\;7,\;11,\;15,\;22,\;30,\;\ldots \\

$-\Id$ & 
4,\;1,\;4,\;1,\;4,\;1,\;4,\;1,\;4,\;1,\;4,\;\ldots & 
4,\;11,\;28,\;63,\;132,\;264,\;504,\;928,\;1660,\;\ldots \\

$S$ & 
2,\;3,\;2,\;1,\;2,\;3,\;2,\;1,\;2,\;3,\;2,\;\ldots & 
2,\;6,\;12,\;25,\;46,\;86,\;148,\;255,\;420,\;\ldots \\[4pt]

$TS$ & 
1,\;2,\;2,\;2,\;1,\;1,\;1,\;2,\;2,\;2,\;1,\;\ldots & 
1,\;3,\;5,\;10,\;15,\;27,\;40,\;66,\;97,\;\ldots \\[4pt]

$-TS$ & 
3,\;3,\;1,\;3,\;3,\;1,\;3,\;3,\;1,\;3,\;3,\;\ldots & 
3,\;9,\;20,\;45,\;90,\;176,\;324,\;585,\;1017,\;\ldots \\[4pt]

$T$ & 
1,\;2,\;3,\;4,\;5,\;6,\;7,\;8,\;9,\;\ldots & 
1,\;3,\;6,\;13,\;24,\;48,\;86,\;160,\;282,\;\ldots \\[4pt]

$-T$ & 
4,\;2,\;4,\;3,\;4,\;4,\;4,\;5,\;4,\;6,\;4,\;7,\;4,\;8,\;4,\ldots & 
4,\;12,\;32,\;77,\;172,\;366,\;744,\;1460,\;2780,\;\ldots \\[4pt]

$T^{2}$ & 
2,\;4,\;6,\;8,\;10,\;12,\;14,\;16,\;18,\;\ldots & 
2,\;7,\;18,\;47,\;110,\;258,\;568,\;1237,\;2600,\;\ldots \\[4pt]

\end{tabular}
\end{table}
\end{example}

\printbibliography
\end{document}